\documentclass[sn-mathphys,Numbered]{sn-jnl}% Math and Physical Sciences Reference Style
\usepackage{multirow}%
\usepackage{amsmath,amssymb,amsfonts}%
\usepackage{amsthm}%
\usepackage{mathrsfs}%
\usepackage[title]{appendix}%
\usepackage{xcolor}%
\usepackage{textcomp}%
\usepackage{manyfoot}%
\usepackage{booktabs}%
\usepackage{algorithm}
\usepackage{algorithmicx}
\usepackage{algpseudocode}%
\usepackage{listings}%
\usepackage{lineno,hyperref} \modulolinenumbers[5]
\usepackage{times} \usepackage{pst-node} \usepackage{tikz-cd}
\usepackage{afterpage}\usepackage{array}
\usepackage{pstricks}
\usepackage{doi}
\usepackage{graphicx}
\usepackage{svg}

\theoremstyle{thmstyleone}%
\newtheorem{theorem}{Theorem}[section]% meant for sectionwise numbers
\newtheorem{corollary}{Corollary}

\theoremstyle{thmstyletwo}%
\newtheorem{remark}{Remark}%

\theoremstyle{thmstylethree}%
\newtheorem{definition}{Definition}%

\newcommand{\nat}{\ensuremath{\mathbb N }}
\newcommand{\zah}{\ensuremath{\mathbb Z }} 
\newcommand{\co}{\ensuremath{\mathbb C }} 
\newcommand{\F}{\ensuremath{\mathcal F }} 
\newcommand{\con}{\ensuremath{\mathbb{C}^{n}}}
\newcommand{\hol}{\ensuremath{\mathcal{H} }}
\newcommand{\re}{\ensuremath{\mathbb R }}

\newcommand{\cp}{\ensuremath{\mathbb{CP}}}

\newcommand{\sph}{\ensuremath{\mathbb{S}}}
\newcommand{\hk}{\ensuremath{\mathbf{H}}}
\newcommand{\hkap}{\ensuremath{\mathbf{H}_{\kappa}}}
\newcommand{\hd}{\ensuremath{\mathbf{H}_{d}}}

\begin{document}
\title[Errett Bishop theorems on Complex Analytic Sets: Chow's Theorem Revisited, and Foliations of Compact Leaves on Kähler Manifolds]{Errett Bishop theorems on Complex Analytic Sets: Chow's Theorem Revisited and Foliations of Compact Leaves on Kähler Manifolds}

%\address{Unidad de Posgrado, UNAM, Mexico, Cto. de los Posgrados, C.U., Coyoacán, 04510 Mexico City}
%\address{IMATE Cuernavaca, UNAM, Mexico, AvUniversidad, 62210 Cuernavaca, Morelos}

\author*[1,2]{\fnm{Verjovsky} \sur{Alberto}}\email{alberto@matcuer.unam.mx}

\author*[2,3]{\fnm{Martínez} \sur{Carlos Eduardo}} \email{cmartineza@ciencias.unam.mx}

\affil*[1]{\orgdiv{IMATE Cuernavaca}, \orgname{UNAM}, \orgaddress{\street{Av Universidad}, \city{Cuernavaca}, \postcode{62210}, \state{Morelos}, \country{Mexico}}}

\affil[2]{\orgdiv{Unidad de Posgrado}, \orgname{UNAM}, \orgaddress{\street{Cto. de los Posgrados, C.U.}, \city{Mexico City}, \postcode{04510}, \state{CDMX}, \country{Mexico}}}


\abstract{
  In this paper, we present a series of seemingly unrelated results of Complex Analysis which are, in fact, connected via a different approach to their proofs using the results of Errett Bishop of volumes, extensions, and limits of analytic varieties.
  We start with a brief introduction to the tools developed by Bishop and show their usefulness by proving Chow's theorem via a technique suggested a long time ago in a beautiful book by Gabriel Stolzenberg, then we show some of the relationships between the theory of analytic subsets and classical results of complex-analytic functions. 
  We finish with the original contributions of this paper which consists of applications of these tools to the theory of holomorphic foliations with alternative and, we believe, simpler proofs to Edwards, Millet, and Sullivan's impactful result for foliations with compact leaves in the case of complex foliations in Kähler manifolds and J. V. Pereira's global stability result for holomorphic foliations on compact Kähler manifolds.
}

\maketitle
\tableofcontents
\section{Introduction}
\noindent The purpose of this paper is first to survey  Errett Bishop's results on 
the limit and the continuation of analytic sets with bounded volume.
We list these statements here as theorems \ref{bishop mapping}, \ref{bishop sequence}, and \ref{bishop}. These theorems were originally published in 1964 in an article titled \textit{Conditions for the Analyticity of certain sets} (see \cite{Bishop}).
The presentation of the results we show here are expository in nature and are inspired by the 1966 Gabriel Stolzenberg's book \textit{Volumes, Limits and Extensions of Analytic Varieties} (see \cite{Stolzenberg}) which also was the catalyst for the exploration of Chow's theorem and stability of holomorphic foliations on Kähler manifolds via Bishop's results.

In 1949, Wei-Liang Chow published the paper \textit{``On Compact Complex Analytic Varieties''}, which contains the proof of a long-standing conjectured result that gives a deep link between analytic geometry and projective algebraic geometry.
The result, now known as \emph{Chow's theorem} states that every closed analytic subvariety of the complex projective space $\cp^n$ is algebraic. The history regarding this result is related to the development of important results of analytic sets of the 20th century. The original proof of Chow was done via analytic simplexes (an analytic simplex is a topological simplex in $\cp^n$ that is biholomorphic with the standard simplex in some complex affine space).
Later in 1953, Reinhold Remmert and Karl Stein proved that the closure of a purely $k$-dimensional is also an analytic set under conditions which, in particular, imply that the complex cone of a projective subvariety is also an analytic set.
Jean-Pierre Serre in his famous 1956 paper \textit{Géometrie Algébrique et Géométrie Analytique} colloquially known as GAGA, proves the same statement with a totally different set of tools.

The proof of Chow's Theorem we present here does not reference directly the algebraic properties of the ring of germs of analytic functions, nor do we use properties of quasi-coherent sheaves; instead, we are concerned with local properties of varieties and the Hausdorff measure of said variety, respectively.
We think this approach is very attractive and easier to understand for people who are not very familiar with the modern categorical language of algebraic geometry. Furthermore, the theory presented here is linked to areas of mathematics that are not usually associated with Chow's result.
In addition, Bishop's results imply both Chow's and Remmert-Stein's theorems directly, meaning that this view is just as efficient and profound as Remmert-Stein's proof.

We present more results that link the theory of complex algebraic sets and some well-known theorems in complex analysis, see \textbf{Table} \ref{table_complex_analysis}.
Furthermore, a careful study of Bishop's results pointed us to apply them to foliations with compact leaves on Kähler manifolds in a manner similar to Edwards, Millet, and Sullivan in \cite{EMS} and J. V. Pereira's result in \cite{Pereira}.

\section{Preliminaries}
 We begin with some basic definitions and terminology which can be found, for instance,  in \cite{Chirka}.

\begin{definition} Let $M$ be a complex manifold. A subset $A \subset{M}$ is called a \textit{complex analytic subset} of $M$, if for each  $p\in{M}$ there exists an open neighborhood $U\subset{M}$ of $p$ and finitely many holomorphic functions $f_1,\ldots, f_k :U\to\co$ such that
\[
  A\cap U=\left \{z\in U \,|\, f_1(z)= \cdots =f_k(z) =0\right\}.
\]
\end{definition}

\begin{definition}
A subset $A \subset{M}$ is called a \textit{complex analytic set}, if for each $p \in A$ there exists an open neighborhood $U$ of $p$ and finitely many holomorphic functions
$f_1,\ldots,f_k : U\to\co$ such that $A\cap U=\left \{z\in U\, :\, f_1(z)=\cdots=f_k(z) =0\right\}$
\end{definition}

\begin{remark}The subtle difference in the definitions ($p\in{M}$ {\it vs} $p\in{A}$) is that a complex analytic subset of a complex manifold is a closed subset of the manifold.
A complex analytic set is locally closed but not necessarily closed. For example, a non-empty open subset of $\co^n$ is an analytic manifold, but it is not closed in $\co^n$.
\end{remark}

We denote the sheaf of holomorphic functions on $\con$ by $\hol_{n}$ and for $U\subset\con$ an open subset, the ring of holomorphic functions on $U$ will be denoted by $\hol_{n}(U)$ or simply $\hol(U)$, thus $\hol_n:=\hol(\con)$.
Given an open subset $U$ of $\con$ and a finite subset of functions $\lbrace f_1,\dots,f_k\rbrace\subset\hol(U)$ we will denote the \textbf{\emph{vanishing locus}} of $f_i$, i.e. the set of points where all of these functions vanish, by
\[\label{vanishing-locus}
  V_U(f_1,\dots,f_k):=\big\lbrace z\in U\,\big|\,f_1(z)= \cdots = f_{k}(z)=0 \big\rbrace
\]

In the case of $\con$, an \textit{\textbf{analytic set}} of $\con$ is a locally ringed subset
$X\subset\con$, meaning that for every $x\in X$ there is a neighborhood $U$ of
$x$ and a finite subset $\lbrace f_1,\dots,f_k\rbrace\subset\hol(U)$ with
\hbox{$X\cap U=V_U(f_1,\dots,f_k)$}, together with the local ring
\[
  \hol_X:= \hol_n/\mathcal{I}(X)\hspace*{0.1cm},\text{where}\hspace*{0.2cm}
  \mathcal{I}(X):=\big\lbrace f\in\hol_n\,|\,f|_{U\cap X}=0 \big\rbrace.
\]

\begin{definition} An \textit{\textbf{analytic space}} $(X,\hol_{X})$ is a
        topological Hausdorff space $X$ together with a local ring structure $\hol_X$ that is locally isomorphic to an analytic set of $\con$. Meaning that, there exists a local homeomorphism $\varphi$ such that its pullback $\varphi_{*}[f]=f\circ\varphi$ is a ring isomorphism between the aforementioned local rings.
        We call a neighborhood with its local isomorphism a \textit{chart}.
\end{definition}

An \textbf{\textit{analytic subspace}} of an analytic space
$(X,\hol_{X})$ consists of a subset $Y\subset X$ such that for every $y\in Y$
there is a chart $(U,\varphi)$ around $y$ with
\[
  \varphi(Y\cap U)=V_{\varphi(U)}(\eta_1,\dots,\eta_k),
\]

\noindent and $\lbrace\eta_1,\dots,\eta_k\rbrace\subset\hol_{n}$, naturally we have a local ring structure $(Y,\hol_{Y})$ given by
\hbox{$\hol_Y:=\hol_X/\mathcal{I}(Y)$}, as before.

\begin{definition} A subset $N\subset{M}$ of the complex $n$-dimensional manifold 
$M$ is called a complex submanifold of  $M$ of dimension $d$ ($0\leq{d}\leq{n}$), if it is closed  and for each
$p\in{N}$ there is an open neighborhood $U\subset{M}$ of $p$ and a biholomorphic map $\varphi:U\to{V}\subset\co^n$, where $V$ is an open
subset (so that $\varphi$ is a holomorphic chart), 
so that.  $\varphi(N\cap{U})=\varphi(U)\cap\co^d$.
\end{definition}
 \medskip
\begin{remark} It follows from the definition that a complex submanifold of a complex manifold \(M\) is an analytic subset of \(M\).
\end{remark}

\begin{remark} Given $M$ a complex manifold. An {\bf immersed complex submanifold} of $M$ is a subset $N$ endowed with a topology (not necessarily the subspace topology) with respect to which it is a topological manifold, and a complex structure with respect to which the inclusion map $N\hookrightarrow{M}$ is a holomorphic immersion.
In other words, $N$ is the image of a holomorphic injective immersion of a complex manifold. Since embeddings with closed image are the same thing as proper injective immersions, the image of a complex manifold under a proper injective immersion is a complex submanifold.
\end{remark}

We denote \textit{complex projective space of dimension $n$} by $\cp^n$; this is a complex manifold and therefore an analytic space with its usual affine charts.
Algebraic subsets on $\mathbb{C}^{n+1}$ consist of the zeroes of a
finite number of polynomials in $\co[z_0,\dots ,z_n]$.
Note that using projective coordinates, the set of vanishing points of homogeneous polynomials is well defined on $\cp^{n}$ and therefore algebraic subsets of $\cp^{n}$ are well defined.
Analytic and algebraic subsets of $\cp^n$ are going to be called \textit{\textbf{projective} analytic subsets} and \textit{\textbf{projective} algebraic subsets} respectively.
The following theorem states the equality of the notions of projective analytic subsets and projective algebraic subsets:
\medskip

\begin{theorem}[Chow]
  Every closed projective analytic subset is a projective algebraic subset.
\end{theorem}

In order to prove this statement, we proceed with some properties of volumes of
purely $k$-dimensional analytic subsets and their Hausdorff measures, then
we proceed to enunciate one of the key theorems for this proof, a result that
we call \textit{\textbf{Bishop's proper mapping theorem}}.

\section{Volumes of analytic sets and Wirtinger's Inequality}

Volumes and metrics have a natural relationship that we are going to exploit throughout this article.
In the context of complex analytic spaces and manifolds, Kähler geometry is special for its interplay
between the metric and the associated volumes, both for the whole manifold and
for its submanifolds. Related to this is the result due to Wirtinger that, among
other things, implies that Kähler submanifolds of a Kähler manifold minimize the
volumes of their respective homological class. This will be of crucial
importance when dealing with the volume function of the leaves on a compact
Kähler manifold in theorem \ref{kahlerEMS}.
\begin{definition}
        Let $M$ be a complex manifold with an integrable almost-complex structure 
        $J:TM\rightarrow TM$. A \textit{Hermitian metric} on $M$ is a smooth family
        real bilinear forms $\{h_{p}\,:\,p\in M\}$ where
        $h_{p}:T_{p}M\times T_{p}M\rightarrow\co$ have the following properties
\begin{itemize}
        \item $h_p(JX,JY)=h_p(X,Y)\hspace{0.3cm}\forall\lbrace X,Y\rbrace\subset T_{p}M\hspace{0.3cm}\forall p\in M$.
        \item $h_p(X,JX)>0\hspace{0.3cm}\forall X\in T_p M\setminus\lbrace 0\rbrace\hspace{0.3cm}\forall p\in M$.
        \item $h_p(X,Y)=\overline{h_p(Y,X)}\hspace{0.3cm}\forall\lbrace X,Y\rbrace\subset T_{p}M\hspace{0.3cm}\forall p\in M$.
\end{itemize}
We call a manifold with a Hermitian metric a \textbf{\textit{Hermitian manifold}}.
\end{definition}
 We often use the standard Riemannian notation $\langle\cdot,\cdot\rangle$ for $h(\cdot,\cdot)$. We note here that a Hermitian metric $h$ has an associated Riemannian metric and an associated 2-form given by the real and imaginary part of $h$ and vice versa. A Riemannian manifold of even dimension $(M,g)$ with a complex structure $J$ has a natural Hermitian metric given by
\[
  h(X,Y):=g(X,Y)-i\omega(X,Y).
\]
Where $\omega$ is its associated 2-form given by $\omega(X,Y)=g(JX,Y)$, which is clearly antisymmetric since $J^2=-Id_M$.
\begin{definition}
        Let $(M,h)$ be a Hermitian manifold with Hermitian metric
        $h$ and associated 2-form $\omega$, $h$ is a Kähler metric if
        $d\,\omega=0$. We call a manifold with a Kähler metric a \textit{Kähler manifold}.
        The associated 2-form will be called a \textit{Kähler form} or a
        \textit{Kähler symplectic form}.
\end{definition}
Being a Kähler manifold has some strong topological restrictions. For example, powers of the Kähler
form $\omega^k$ are nontrivial representatives of cohomology classes in
$H^{2k}(M ;\re)$, meaning that these groups are never trivial. More important
for our purposes is the fact that, by our definition of submanifold, it follows that every complex submanifold
(or immersed complex submanifold) of a Kähler manifold is also Kähler. This follows from the fact that the Kähler form $\omega$ is closed and the commutativity of the exterior derivative with the pullback $f^*$, where $f$ is either an immersion or a parametrization. Other consequences for Kähler manifolds are the following:

 \medskip

\begin{theorem}[Wirtinger's inequality]Let $V$  be a real vector space  of even dimension $2k$ endowed with a positive-definite inner product $g$, symplectic form $\omega$, and  an almost-complex structure $J$  linked, as before by $\omega(u, v) = g(J(u), v)\, ,\forall\,u,v\in{V}$. Then for any orthonormal vectors $u_1, \dots, u_{2k}$ the following inequality holds:
 \[ (\underbrace {\omega \wedge \cdots \wedge \omega } _{k{\text{ times}}})(u_{1},\ldots ,u_{2k})\leq k!.
 \]
There is equality if and only if the span of $u_1,\dots, u_{2k}$ is closed under the action of $J$.
\end{theorem}
As a consequence of the volume forms for submanifolds of Kähler manifolds we have the following
\begin{corollary}[Complex submanifolds of Kähler manifolds minimize volume]
Let $M$ be a Kähler manifold with Kähler form $\omega_M$ and let $f:N\rightarrow M$ be a closed and oriented immersion
of an oriented real manifold of real dimension $2k$. Let $\omega=f^*\,\omega_m$, then
\begin{equation}
            \int_{N} \frac{\omega^k}{k!}\leq \int_N d\textrm{Vol}_N\hspace{0.3cm}\text{where }d\textrm{Vol}_N\text{ is the volume form of N,}
\end{equation}
and the equality holds if and only if $N$ is a complex submanifold of $M$.
\end{corollary}
We can state the natural conclusions as the following theorem and its corollary
\begin{theorem}%%Teorema o Corolario?
        Any complex submanifold of a Kähler manifold is a minimal submanifold.
\end{theorem}
 \medskip
\begin{corollary}\label{wirtinger}
        Let $N$ be a complex compact submanifold of a Kähler manifold $M$, then $N$ is a volume-minimizing submanifold in its homology class $H_{2k}(M,\partial N,\zah)$, meaning that for any real submanifold $X$ of real dimension $2k$ that is homologous to $N$, has a greater volume than $N$, i.e.
        \[
          \textrm{Vol}_{2k}(N)\leq \textrm{Vol}_{2k}(X).
        \]
\end{corollary}
Given $(M,g)$ an oriented Riemannian manifold of dimension $n$, then $M$ has a natural volume form $\textrm{Vol}_{n}$. In local coordinates, it is expressed as
\[
 \textrm{Vol}_{n}=\sqrt{|g|}dx_{1}\wedge\dots\wedge dx_{n}.
\]
Let $\omega=-Im\langle\cdot,\cdot\rangle$ be the standard 2-form of the standard Euclidean Kähler metric in $\con$ and let $M$ be a complex submanifold of $\con$ by Wirtinger's inequality, if $\textrm{Vol}_{2k}(M)$ is the volume of $M$ given by the Riemannian structure as before, $g=Re\langle\cdot,\cdot\rangle|_{M}$, then
\begin{equation}
        \textrm{Vol}_{2k}(M)= \frac{1}{k!}\int_M \omega^k.
\end{equation}
If $X$ is a purely $k$-dimensional analytic subset of $\con$, the singular locus of $X$, $\Sigma(X)$ is the set of points where the defining functions of $X$, say $X\cap U=V(f_{1},\dots,f_{n-k})$ have derivatives that span a space of dimension \emph{less than $k$}, meaning that the matrix defined by the partial derivatives $A=(\partial_{z_{i}}f_{j})_{ij}$ has a rank \emph{less than $k$}. Therefore $\Sigma(X)$ is an analytic subset of $X$ of lesser dimension defined as the \textbf{\emph{locus}} of the minors of dimension \emph{greater than} $n-k$ of $A$ and so $M=X\setminus\Sigma(X)$ is a complex manifold with the same volume as $X$, meaning
\[
    \textrm{Vol}_{2k}(X)= \frac{1}{k!}\int_{X} \omega^k= \frac{1}{k!}\int_{X\setminus\Sigma(X)} \omega^k=
    \textrm{Vol}_{2k}(X\setminus\Sigma(X)).
\]
\section{Calibrated manifolds}
 Now we would like to address \emph{calibrations and calibrated manifolds}, a \textbf{\emph{calibrated geometry}} is a way to describe a geometry via a distinguished family of submanifolds, more precisely, a volume minimizing family of subsets. This approach can be better suited for understanding \textbf{\emph{foliated geometries}}, as we will show that Bishop's theorems will help us with such an endeavor. In order to distinguish said family of subsets, a \textbf{\emph{calibration}} is needed; a calibration can be thought of as a way to find volume-minimizing manifolds of certain dimensions (see \cite{Harvey}). Kähler manifolds and their submanifolds have a natural calibration by Wirtinger's inequality. The relationship between the symplectic and volume forms of a hermitian inner product and the Riemannian volume implies that the normalized exterior powers of the Kähler form of a Kähler manifold are calibrations.
\newpage
\begin{definition}[Calibrated manifolds] A \textbf{calibrated $n$-manifold} is a Riemannian manifold $(M, g$) of dimension $n$ equipped with a differential $p$-form $\varphi\,$  ($0 \leq {p} \leq{n}$) with the following properties
\begin{enumerate}
 \item $\varphi$ is closed, meaning $d\varphi = 0$, where $d$ is the exterior derivative.
 \item For any $x \in M$ and any oriented $p$-dimensional subspace $\xi$ of $T_{x}M$, $\varphi|_{\xi} = \lambda \textrm{Vol}_{\xi}$ with $\lambda\leq 1$.
\end{enumerate}
When referring to Riemannian manifolds, $\textrm{Vol}_{\xi}$ denotes the volume form of  $\xi$ with respect to the Riemannian metric $g$. The $p$-form its called a $p$-\textbf{calibration}.
\end{definition}
For $x\in M$, set $G_{x}(\varphi) = \{ \xi \subset T_{x}M\,|\,  \varphi|_{\xi} = \textrm{Vol}_{\xi} \}$. In order for the theory to be nontrivial, we need $G_x(\varphi)$ to be nonempty. Let $G(\varphi)$ denote the union of all $G_x(\varphi)$ with $x\in M$.
\begin{definition}[Calibrated submanifold]
A $p$-dimensional calibrated submanifold of a manifold $M$ with calibration $\varphi$ is an oriented submanifold $\Sigma$ such that the calibration restricted to the tangent bundle of $\Sigma$, equals the induced volume form of the submanifold $\varphi|_{T\Sigma}=\textrm{Vol}_{\Sigma}$. Equivalently $T\Sigma\subset G(\varphi)$
\end{definition}
For calibrated manifolds and submanifolds the volume-minimizing property of submanifolds in the same homological class is proven by the following one line argument
\begin{equation}\label{oneline}
  \int_{\Sigma}\mathrm{Vol}_{\Sigma}=\int_{\Sigma}\varphi =\int_{\Sigma'}\varphi \leq \int_{\Sigma'}\mathrm{Vol}_{\Sigma'}
\end{equation}
where the first equality holds because $\Sigma$ is calibrated, the second equality is Stokes' theorem (as $\varphi$ is closed), and the last inequality holds because $ \varphi$ is a calibration.
Here $\Sigma$ is a calibrated submanifold of $M$ and $\Sigma'$ is any submanifold in the same homology class of $\Sigma$. 

\begin{remark}
Note that a Kähler manifold $(M,\omega)$ is a calibrated manifold with its calibrations given by the Kähler form and its powers $\{\omega^{k}\,|\, 1\leq k\leq \dim(M)\}$. Meaning that $(M,\omega^{k})$ is a calibrated manifold for every $k\in\{1,\dots,\dim(M)\}$, and the previous one-line argument is a version of Wirtinger's inequality. This has the important consequence that two isotopic compact complex submanifolds of a Kähler manifold have the same volume. We will show in Theorem \ref{kahlerEMS} that the volume function
of a holomorphic foliation by compact leaves on a connected Kähler manifold is constant
in an open Zariski dense set, which is the complement of the union of the leaves with nontrivial holonomy, which must be a finite group, and,  by a theorem of H. Cartan \cite{Cartan}, this must be an analytic set.
\end{remark}
\section{Hausdorff measure and Bishop's theorems}

\noindent Besides the volumes of analytic sets, we will study the Hausdorff measure of said sets to determine their dimension. It is worth noting that for purely dimensional analytic sets, the topological and Hausdorff dimensions coincide. Moreover, an advantage of analyzing analytic sets from the point of view of Bishop's results is that, under the right assumptions, it provides guarantees about the Hausdorff dimension of limits and closures of said analytic sets, as we shall see.

We denote the $\delta$-Hausdorff measure of a subset $S\subset\con$ by $\hd(S)$, see \citep{Stolzenberg}[ch. 3]. Here is a brief introduction to it, along with some of its key properties.

Let $(X,\rho)$ be a metric space and let $S \subset X$, then let the diameter of $S$ be defined by $\textrm{diam}(S) = \sup\{\rho(x,y)\,|\, x, y\in S\}$. That is to say, the diameter of a set is the distance between the farthest two points in the set.
\begin{definition}
Let $S$ be any subset of $X$, and $\delta > 0$ a real number. We define the Hausdorff outer Measure of dimension $d\in\re^{+}$ bounded by $\delta$ (written $H^{d}_{\delta}$) by:
\begin{equation}
H^{d}_{\delta}(S)=\inf\Big\{ \sum^{\infty}_{i=1}{(\textrm{diam}(U_{i}))}^{d}\,|\,S\subset\bigcup^{\infty}_{i=0}U_{i},\quad\textrm{diam}(U_{i}) < \delta\Big\}. %%Checa que onda
\end{equation}
Where the infimum is taken over all countable covers of $S$ by sets $U_{i} \subset X$ satisfying $\textrm{diam}(U_{i})<\delta$.
\end{definition}
If we allow $\delta$ to approach zero, the infimum is taken over a decreasing collection of sets, and therefore $H^{d}_{\delta}$ increases. We can conclude that
\[
  \hd(S):=\lim_{\delta\rightarrow 0} H^{d}_{\delta}(S)=\sup_{\delta>0}H^{d}_{\delta}(S)
\]
exists, but may be infinite. We call this limit the \emph{\textbf{Hausdorff Outer Measure} of dimension} $d$, the Hausdorff Outer Measure is a way for us to measure dimensionality is a way that is directly related to volumes of analytic sets a we will show shortly. Here we list some commonly known properties that are useful for our purposes and can be found in standard sources see \cite{Stolzenberg}.

\begin{enumerate}
        \item[1.] If $\hd(S)<\infty$, and $d<\kappa$, then $\hkap(S)=0$.

        \item[2.] If $f:X\rightarrow Y$ is a Lipschitz continuous function with
        Lipschitz constant $\lambda$, then for any $\delta\in\re^{+}$ and
        $S\subset X$, the following inequality holds
            \[
                \hd(f(S))\leq\lambda^{\delta}\hd(S).
            \]

        \item[3.]If $X=\re^{n}$ and $S=M$ is a smooth submanifold of dimension
        $k\in\zah^{+}$, then the volume of $M$ as a submanifold is related to
        its Hausdorff measure by the formula
        \[
            \textrm{Vol}_{k}(M)=\alpha_k\hk_k(S),\hspace*{0.3cm}\alpha_k= \frac{\pi^{k/2}}{\Gamma(k/2+1)},
        \]
\noindent where $\Gamma(z)$ is the Euler's gamma function. The constant $\alpha_k$ is the volume of the unit ball in $\re^k$
\end{enumerate}

Related to the Hausdorff measure is the Hausdorff metric defined for $K_1$ and $K_2$ compact subsets of a metric space $(X,d)$ as

\[ d_{h} (K_1,K_2)=\max \left\{\,\sup _{x\in K_1}d(x,K_2),\,\sup _{y\in K_2}d(K_1,y)\,\right\}.
\]

The Hausdorff metric allows us to define a convenient notion of convergence of closed subsets; let $\left\{S_n \right\}$ be any
sequence of closed subsets of $X$, then we say that $S_n$ converges to the closed subset $S$ in the Hausdorff metric, written as $S_n\overset{h}\rightarrow S$,
if for every $K\subset X$ compact we have
\[
d_h(K\cap S_n,K\cap S)\rightarrow 0.
\]

The first of Bishop's results that we present here is very useful for understanding some of the analytical properties
of limits (as defined before) of purely $k$-dimensional analytical subsets.

\begin{theorem}[Sequence theorem]\label{bishop sequence} Let $\lbrace
        V_n\rbrace$ be a sequence of purely $k$-dimensional subsets of a
        domain $\Omega\subset\con$ such that $V_n\overset{h}\rightarrow V$, with
        $V\subset\Omega$ a closed subset. If
        \[
        \textrm{Vol}_{2k}(V_n)\leq \rho\hspace*{0.2cm}\forall n\in\nat,
        \]
        \noindent for $\rho\in\re$ a positive constant, then for the Hausdorff measure we have $\hk_{2k+1}(V)=0$,
        moreover, $V$ is a purely $k$-dimensional analytic subset of $\Omega$.
\end{theorem}

As a direct application of this result, one can show the following very useful proposition (see \citep{Stolzenberg}[ch. 4]).

\begin{theorem}[Bishop's proper mapping theorem]\label{bishop mapping} Let
        $\Omega\subset\con$ be a domain that contains $0$ and let
        $S\subset\Omega$ be a closed subset. If $\hk_{2k+1}(S)=0$, then there
        is a suitable coordinate change of $\con$, say, $(z_1,\ldots,z_n)$ and a
        neighborhoods, $\Omega_k\subset\co^k$ and $\Omega_{n-k}\subset\co^{n-k}$, such that
        $0\in\Omega_k\times\Omega_{n-k}\subset\Omega$ and the projection
        \[
        \pi_k:S\cap(\Omega_k\times\Omega_{n-k})\rightarrow\Omega_k\, ,\hspace{0.3cm}\pi_k(z,w):=z,
        \]
        \noindent is a proper map.
\end{theorem}
\begin{remark} When \(S\) is a purely $k$-dimensional analytic subset and not just a closed subset, then the following theorem implies that the projection \(\pi\) is a ramified analytic covering.

\begin{theorem}\label{Remmert}
  Let \(\Omega=\Omega_{k}\times\Omega_{n-k}\) be a open subset of \(\con\) and \(\pi:\Omega\rightarrow\Omega_{k}\) the projection \((z_{k},z_{n-k})\mapsto z_{k}\). Let \(A\subset\Omega\) be a analytic subset such that \(\pi:A\rightarrow\Omega_{k}\) is proper. Then \(A^{'}=\pi(A)\) is an analytic subset in \(\Omega_{k}\) and \(|\pi^{-1}[z_{k}]\cap A|\) is locally finite in \(\Omega\).
\end{theorem}
See \cite{Chirka}[pp. 47]. With this, the regular points of an analytic set can be characterized.
\end{remark}

\begin{corollary}
  Let \(0\in A\subset\con\) be an analytic set. The point \(0\in A\) is a regular point if and only if there is a open set \(0\in U\subset\con\) and a coordinate plane at \(\co_{k}\) such that the projection \(\pi_{l}:A\cap U\rightarrow\co_{k}\cap U\) is one to one.
\end{corollary}
\begin{remark} Remmert's Proper Mapping Theorem, which states that if $f:X\to{Y}$ is a proper holomorphic map between complex analytic spaces, then the image $f(X)\subset{Y}$  is an analytic subspace is a consequence of Theorem \ref{Remmert} by Bishop.
Apart from its clarity, Bishop's proof is constructive in the sense that it replaces abstract ideal and sheaf theory with direct analytic function bounds and approximations.

\end{remark}
  \section{Consequences of Bishop's proper mapping theorem}

\noindent As mentioned, this result by Bishop can be used to prove many other important
results (see \cite{Stolzenberg}), one of the most significant is the proof of
Remmert-Stein's theorem, which was generalized and proved by Bishop in
\cite{Bishop}.

\begin{theorem}[Remmert-Stein]\label{Rem-Stein}
        Let $\Omega\subset\con$ be an open subset and $Y$ an analytic subset
        of $\Omega$ and let $X$ be a analytic subset of
        $\Omega\setminus Y$. If $Y$ is of dimension at most $k-1$ and $X$ is of pure dimension $k$,
        then the closure of $X$ in $\Omega$, $\overline{X}\cap\Omega$, is an analytic subset of $\Omega$.
\end{theorem}
This is an essential step toward the proof of Chow's theorem if one is trying to avoid using categorical methods and
quasi-coherent sheaves. This is because Remmert and Stein's result implies that the $\textrm{Cone}(X)$ of a projective analytic
subset of dimension $k$, $X\subset\cp^n$ is an analytic subset of dimension $k+1$ in $\co^{n+1}$, where the cone is defined by
\begin{equation}
    \textrm{Cone}(X):=\pi^{-1}[X]\cup\lbrace0\rbrace,\hspace{0.2cm}\pi:\co^{n+1}\setminus\lbrace 0\rbrace\rightarrow\cp^n.
\end{equation}
\noindent Here $\pi$ is the usual projection of $\co^{n+1}\setminus\lbrace 0\rbrace$ onto the
projective space, so clearly $\textrm{Cone}(X)=\overline{\pi^{-1}[X]}$, and since $\pi$
is an analytic projection, $\pi^{-1}[X]$ is an analytic subset. Then from this
point on, the classical proof is to use the fact that the cone is homothetic-invariant to
show that the ideal of locally defined holomorphic functions that vanish at the
cone has a countable basis. Then, with Hilbert's basis theorem, it is easy to prove the
fact that the ring of germs of holomorphic functions is Noetherian. This shows
that $\textrm{Cone}(X)$ is in fact algebraic, see \cite{Chirka}, but the same result can
be proved without algebraic methods with an equally simple proof using only
the geometric and analytical tools we have presented thus far. We obtain this result from the following consequence of Bishop's proper mapping Theorem.
\begin{theorem}[Bishop]\label{bishop}
        Let $V$ be a purely $k$-dimensional subset of $\con$ and
        let $B(R,0)$ be the standard ball in $\con$ of
        radius $R$. If there is  a constant $C\in\re^{+}$ such that
        \begin{equation}
                \textrm{Vol}_{2k}(V\cap B(R,0))\leq CR^{2k} \hspace*{0.2cm}\forall\, R\in\re^{+},
        \end{equation} then $V$ is algebraic.
\end{theorem}
\noindent\textbf{Sketch of the proof:} Let $\lbrace R_n\rbrace\subset\re^{+}$
be an unbounded sequence, $R_n\rightarrow\infty$ and define $V_n$ as the image of $V\cap
B(0,R_n)$ by the homothety $z\longmapsto z/R_n$, then $\lbrace V_n\rbrace$ is a
sequence of analytic sets of the unit ball with $\textrm{Vol}_{2k}(V_n)<C$, and such
that $0\in\lim_{n\rightarrow\infty} V_n$. Then by the proper mapping theorem,
there is a neighborhood of $0$, $\Delta=\Delta_k\times\Delta_{n-k}$ such that
the projection on the first factor $\pi_k$ is a $\sigma$-sheeted branched
covering for each $V_n\cap\Delta$, since $R_n\rightarrow\infty$ the balls
$B(0,R_n)$ cover the whole of $\co^{n+1}$ and we deduce that the projection
$\pi_k$ restricted to $V$ is a $\sigma$-sheeted branched covering. From this
let us construct a set of canonically defining functions for $V$; for each
$z\in\con$ we denote
\[
\pi_k^{-1}[\pi_k(z)]\cap V=\lbrace\alpha_1(z),\ldots,\alpha_{\sigma}(z)\rbrace,
\]

\noindent and let $P_{\alpha}:\con\times\con\rightarrow\co$ be

\begin{equation}
    P_{\alpha}(z,w):=\langle z-\alpha_1(z),w\rangle\ldots\langle z-\alpha_{\sigma}(z),w\rangle,
\end{equation}

\noindent we observe that if $z\in V$, $P_\alpha(z,w)=0$, so for $z$ outside of
$V$ we can chose a $w$ such that $\langle z-\alpha_j(z),w\rangle\neq0$ for all
$j$. Now by expanding $P_{\alpha}$ in powers of $w$

\begin{equation}
        P_{\alpha}(z,w)=\sum_{|\mu|\leq\sigma} \eta_{\mu}(z)w^{\mu}
\end{equation}

\noindent then because $\pi_k$ is a branched covering the functions
$\eta_{\mu}$ are analytic. By applying a homothety for a suitable $R_n$, one
can prove using Cauchy's estimates that the functions $\eta_{\mu}$ are in fact
polynomials of degree at most $\sigma-|\mu|$ thereby proving the
theorem.
\section{Proof of Chow's theorem via Bishop's theorems}
    \begin{proof}[Chow via Bishop] By Theorem \ref{Rem-Stein}, $\textrm{Cone}(X):=\overline{\pi^{-1}[X]}$
            is an analytic set of $\co^{n+1}$, and if $X\subset\cp^n$ is of dimension $k$ (real dimension $2k$), then
            $\textrm{Cone}(X)$ has dimension $k+1$. Now the mapping
            $\pi:\sph^{2k+1}\rightarrow\cp^n$ known as \textbf{Hopf's Fibration}, is a \emph{Riemannian fibration} from the sphere with its canonical metric and the \emph{Fubini-Study metric} on projective space, it fibers
            $\cp^n$ into circles (of length $2\pi$), so by Fubini's theorem;
            \[
                \textrm{Vol}_{2k+1}(\textrm{Cone}(X)\cap\sph^{2n+1})= 2\pi\,\textrm{Vol}^{\cp}_{2k}(X):=A,
            \]
            \noindent where $\textrm{Vol}^{\cp}$ is the ``projective volume'',
            meaning the volume form of the complex projective space given by the Fubini-Study metric.
            Clearly, $A$ is finite since $X$ is compact, which means that the volume of the
            intersection $\textrm{Cone}(X)\cap B(R,0)$ grows at most polynomially as
            $R\rightarrow\infty$, since $\textrm{Cone}(X)$ is homothetic-invariant and using
            polar coordinates, we see that
    \[
    \textrm{Vol}_{2k+2}(\textrm{Cone}(X)\cap B(R,0))= \int_{\textrm{Cone}(X)\cap B(R,0)}dr\wedge\sigma_r,
    \]
    \noindent where $\sigma_r$ is the $2k+1$ volume form of
    $S(0,r)\cap{X}$  with
    $S(0,r):=\lbrace\| z\|=r\rbrace\,$ with $0\leq{r}\leq{R}$. We can find a bound for the volume of $\textrm{Cone}(X)\cap B(R,0)$ as follows:

    \[
        \int_{(\textrm{Cone}(X)\cap B(R,0))}dr\wedge\sigma_r=
        \int^R_0 r^{2k+1}dr\,\int_{\textrm{Cone}(X)\cap\sph^{2n+1}}\sigma_1=
        \frac{A}{2k+2}R^{2k+2},
    \]
    \noindent so setting $C=\frac{A}{2k+2}$ by Theorem \ref{bishop} this means
    $\textrm{Cone}(X)$ is algebraic and therefore $X$ is also algebraic.
    \end{proof}
 \begin{figure}[h!]
    \centering \def\svgwidth{180pt} %% Creator: Inkscape 1.0.1 (3bc2e813f5, 2020-09-07), www.inkscape.org
%% Accompanies image file 'cone.pdf' (pdf, eps, ps)
\begingroup%
  \makeatletter%
  \providecommand\color[2][]{%
    \errmessage{(Inkscape) Color is used for the text in Inkscape, but the package 'color.sty' is not loaded}%
    \renewcommand\color[2][]{}%
  }%
  \providecommand\transparent[1]{%
    \errmessage{(Inkscape) Transparency is used (non-zero) for the text in Inkscape, but the package 'transparent.sty' is not loaded}%
    \renewcommand\transparent[1]{}%
  }%
  \providecommand\rotatebox[2]{#2}%
  \newcommand*\fsize{\dimexpr\f@size pt\relax}%
  \newcommand*\lineheight[1]{\fontsize{\fsize}{#1\fsize}\selectfont}%
  \ifx\svgwidth\undefined%
    \setlength{\unitlength}{340.15748031bp}%
    \ifx\svgscale\undefined%
      \relax%
    \else%
      \setlength{\unitlength}{\unitlength * \real{\svgscale}}%
    \fi%
  \else%
    \setlength{\unitlength}{\svgwidth}%
  \fi%
  \global\let\svgwidth\undefined%
  \global\let\svgscale\undefined%
  \makeatother%
  \begin{picture}(1,1)%
    \lineheight{1}%
    \setlength\tabcolsep{0pt}%
    \put(0,0){\includegraphics[width=\unitlength,page=1]{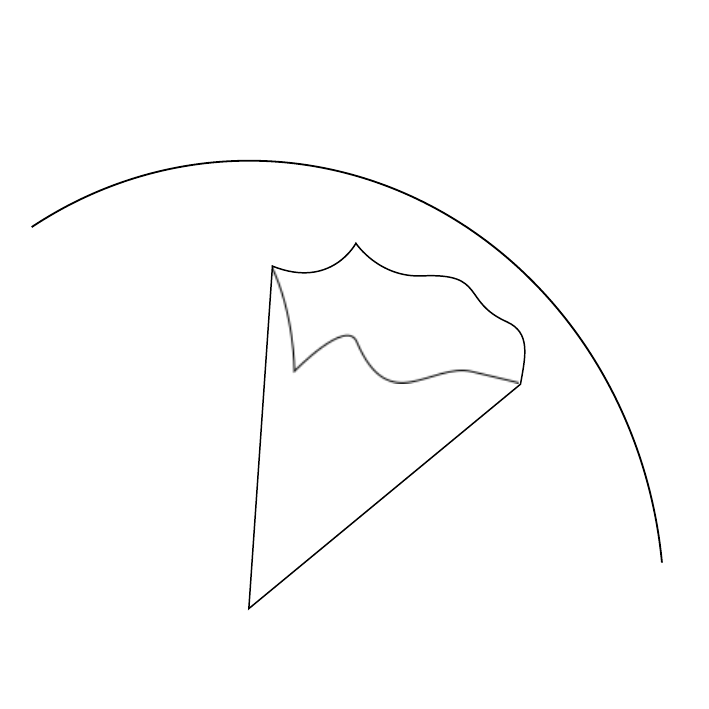}}%
    \put(0.34980729,0.09541632){\color[rgb]{0,0,0}\makebox(0,0)[lt]{\lineheight{1.25}\smash{\begin{tabular}[t]{l}$0$\end{tabular}}}}%
    \put(0,0){\includegraphics[width=\unitlength,page=2]{cone.pdf}}%
    \put(0.73723253,0.38256517){\color[rgb]{0,0,0}\makebox(0,0)[lt]{\lineheight{1.25}\smash{\begin{tabular}[t]{l}$C(X)$\end{tabular}}}}%
    \put(0.08448537,0.75328914){\color[rgb]{0,0,0}\makebox(0,0)[lt]{\lineheight{1.25}\smash{\begin{tabular}[t]{l}$\mathbb{S}^{2n+1}$\end{tabular}}}}%
  \end{picture}%
\endgroup%

    %\centering\includegraphics[width=0.8\textwidth,natwidth=210,natheight=240]{cone.pdf_tex}
    \caption{Cone of $X$ having finite polynomially bounded volume.}\label{fig:cone}
\end{figure}

    \newpage    
         
\section{Proof of Montel's Theorem via Bishop}
We now proceed to prove a generalization of Montel's result of compact subsets of holomorphic functions on the unit disk. We note that there is already another way to prove a version of Montel's result that was proved by Lelong in \cite{Lelong}, although its formulated using the language of integration currents.

\begin{definition}
  A family $F$ of analytic functions on a domain $\Omega\subset\co^{n}$ is \textbf{normal} in $\Omega$ if every sequence of functions $\lbrace{f_n}\rbrace_{n\in\nat}\subseteq{F}$ contains either a subsequence $\lbrace f_{n_k}\rbrace_{k\in\nat}\subset F$  which converges to a limit function $f\not\equiv\infty$,  uniformly on each compact subset of  $\Omega$, or a subsequence  which converges uniformly to $\infty$ on each compact subset.
\end{definition}

The open unit ball on $\con$ will be denoted $B$ with its closure on $\con$ denoted by $\overline{B}$. Let
\[
  \mathcal{A}(B):=C(\overline{B})\cap\hol(B),
\]

\noindent be the Banach algebra of holomorphic functions in $B$ that are
continuous on $\overline{B}$ and therefore also continuous on $S:=\partial B$. We will prove the following result

\begin{theorem}[Montel's Theorem] Let $F\subset\mathcal{A}(B)$, be a set of locally bounded functions, i.e.
for every $z\in B$ there exists an open neighborhood $z\in\Omega\subset B$ such that
\[
  \| f|_{\Omega}\|_{\infty}=\sup\{|f|\,|\,z\in\Omega\}\leq C_{\Omega},\hspace{0.2cm}\,\text{for some }C\in\re^{+},
\]

\noindent then $F$ is a normal set.
\end{theorem}
We are going to prove this by means of Bishop's results. First, we note that the graph of a holomorphic function $f:B\rightarrow\co$ is an analytic subset of pure dimension $n$ on $B\times\co$ since $\Gamma_{f}=V_B(w-f(z))=\{(z,w):w-f(z)=0), z\in B, w\in\co\}$ (\ref{vanishing-locus}), where $\Gamma_{f}$ denotes the graph of $f$, and $z=(z_1,\dots,z_n)\in\con$.
Also if $\{f|_{\Omega}\}$ are bounded for all $f\in F$, then $\textrm{Vol}_{2n}(\Gamma_{f|_{\Omega}})$ have uniformly bounded volume as we shall see.
\begin{proof}
Let $\epsilon\in(0,1)$ and we take
$B_{\epsilon}=\overline{B(1-\epsilon,0)}\subset B$, since
$B_{\epsilon}$ is a compact subset of $B$, we have that $F$ is uniformly
bounded in $B_{\epsilon}$. Let $C\in\re^{+}$ be a bound for $|
f(z)|$ for all $z\in B_{\epsilon}$ and $f\in F$, then by Cauchy's
integral formula for the open ball (see \citep{rudin}[Ch. 3, 3.2.4.]), if
$S_{\epsilon}:=\partial B_{\epsilon}$, then
\begin{equation}
  f(z)=\int_{S_{\epsilon}} \frac{f(\zeta)}{(1-\langle z,\zeta\rangle)^{n}}\,d\sigma(\zeta)
  \hspace{0.3cm}\,\forall z\in B_{\epsilon}\setminus S_{\epsilon},
\end{equation}

\noindent where $\sigma$ is the usual Lebesgue measure in $S_{\epsilon}$. Taking the derivative  with
respect to $z_{j}$ we obtain

\begin{equation}
  \dfrac{\partial f}{\partial z_j}(z)=\int_{S_{\epsilon}} \frac{f(\zeta)n\zeta_j}{(1-\langle z,\zeta\rangle)^{n+1}}\,d\sigma(\zeta)
  \hspace{0.3cm}\,\forall z\in B_{\epsilon}\setminus S_{\epsilon},
\end{equation}
\noindent therefore, we have the following bound for the partial derivatives of
all $f\in F$ and $z\in B_{\epsilon}\setminus S_{\epsilon}$,
\begin{equation}\label{bound derivative}
        \left|\dfrac{\partial f}{\partial z_j}(z)\right|\leq C \frac{n \textrm{Vol}_{2n-1}(S_{\epsilon})(1-\epsilon)}{(1-(1-\epsilon)^2)^{n+1}}.
\end{equation}
Now, to show that the hypotheses of theorem \ref{bishop sequence} are satisfied, we use
this constant bound of the derivatives to show that for a sequence of graphs of
functions $\lbrace f_n\rbrace_{n\in\nat}$, the $2n$ dimensional volumes of
their graphs are uniformly bounded in $B_{\epsilon}$ by a constant
$M\in\re^{+}$. This is because the volume of $\Gamma_f$ is given by
\begin{equation}\label{graph volume}
        \textrm{Vol}_{2n}(\Gamma_{f})=\int_{B_{\epsilon}}|\lambda(z)|\,dx_1\ldots dx_{n}dy_1\ldots dy_n,
\end{equation}
\noindent where $z_j=x_j+iy_j$ and $\lambda:B_{\epsilon}\rightarrow\re$ is given by the pullback of the
parametrization function $\varphi(z)=(z,f(z))$
\[
  \varphi^*\omega=\lambda(z)(dx_1\wedge dy_1)\wedge\ldots\wedge(dx_i\wedge dy_n),
\]
\noindent with $\omega$ being the volume form of $\Gamma_f\subset B\times\co$. A straightforward calculation
shows:
\begin{equation}
        \lambda(z)=\sqrt{1 + \|\nabla f\|^2},
\end{equation}
\noindent with
\[
  \nabla f =\Big(\dfrac{\partial f}{\partial z_1},\dots,\dfrac{\partial f}{\partial z_n}\Big).
\]
Therefore, the volumes of all the graphs are uniformly bounded by (\ref{bound derivative}) and (\ref{graph volume}). Now the only thing left to show is the
convergence with respect to the Hausdorff metric of a subsequence of
$\lbrace\Gamma_{f_n}\rbrace$ in $B_{\epsilon}$. This follows from
the fact that $| f(z)|$ is uniformly bounded in $B_{\epsilon}$ for all
$f\in F$. Since the image of all $f_n$ are inside a compact set, it therefore follows that for
$z\in S_{\epsilon}$ there exists a convergent subsequence $\lbrace
f_{n_k}(z)\rbrace$ and by Cauchy's integral formula $\lbrace f_{n_k}(z)\rbrace$ converges uniformly on $B_{\epsilon}\setminus
S_{\epsilon}$. This means that $\Gamma_{f_{n_k}}\rightarrow\Gamma_f$ in the Hausdorff metric with
$\Gamma_f\subset B_{\epsilon}\times\co$, and by theorem \ref{bishop sequence}
we know $\Gamma_f|_{B_{\epsilon}}$ is an analytic subset of pure complex
dimension $n$. Moreover, it is clearly the graph of a holomorphic function
\hbox{$f:B_{\epsilon}\setminus S_{\epsilon}\rightarrow\co$} since at each fiber
$\lbrace z_0\rbrace\times\co$, the intersection of each graph
$\Gamma_{f_{n_k}}$ gives a unique point $(z_0,f_{n_k}(z_0))$.
If there were two distinct points $(z_0,w_1)$ and $(z_0,w_2)$ at the
intersection $(\lbrace z_0\rbrace\times\co)\cap\Gamma_f$, then taking a
compact set $K$ containing $\lbrace f_{n_k}(z_0)\rbrace\cup\lbrace
w_1,w_2\rbrace$, by Hausdorff convergence we would have that the limit
of $\lbrace f_{n_k}(z_0)\rbrace$ is not unique. Thus
$\Gamma_f$ has to be the graph of a holomorphic function in $B_{\epsilon}$ for all
$\epsilon\in(0,1)$. By analytic continuation, we thus define $f$ in all of $B$. The graph
of $\Gamma_f$ can be defined by 
$\Gamma_f=\left\{(z,w)\,:\,\,F(z,w):=f(z)-w=0,\, z\in{B}, w\in\,\mathbb{C}\right\}$. The graph
is the zero locus of the holomorphic function $F(z,w)=f(z)-w$, since 
$\partial{F}/\partial{w}=1$, the complex gradient of $F$ does not vanish; it follows by the implicit function theorem that $\Gamma_f$ is smooth.
\end{proof}
\begin{remark}
  As previously noted, there is a version of Montel's theorem proved by Lelong in terms of the integration currents defined
  by analytic subsets of pure dimension $k$. An integration current defined by an analytic subset of dimension $k$,
  \(A\) is the linear operator
  \[
    t(A)[\phi] = \int_A\phi,\hspace{0.1cm}\text{where $\phi$ is a differential $(k,k)$ form with compact support}.
  \]
  In this case the boundedness refers to the norm of the aforementioned operators, where we say that a family $F$ of
  purely $k$ dimensional analytic subsets are locally bounded if \hbox{$\{\| t(A)\|\,|\,A\in F\}$}
  are bounded for every compact set. See \cite{Lelong}    
\end{remark}
The following table shows the similarities between the results shown here and others in the classical theory
of holomorphic functions of one variable.
\begin{table}[h!]
        \caption{Similarities between complex analysis and the theory of analytic sets.}\label{table_complex_analysis}
        \centering
        \begin{tabular}{| m{5cm} | m{5cm} |}
                            \toprule
                            \begin{center} %\vspace*{0.2cm}
                                    \underline{\textbf{Complex Analysis}}
                            \end{center} &
                            \begin{center} %\vspace*{0.2cm}
                                    \underline{\textbf{Theory of analytic sets}}
                            \end{center} \\
                            \midrule
                            %\vspace*{0.1cm}
                            \begin{center}
                                    \textit{Liouville's theorem}
                            \end{center} &
                            \begin{center}
                                    \textit{Bishop's theorem (Theorem \ref{bishop})}
                            \end{center}\\
                            \midrule
                                If $| f(z)|\leq C\,R^k$ for $| z|\leq R$ for all $R\in\re^{+}$ with $f$
                                entire, $k$ a positive integer and $C$ a positive constant, then $f$ is a polynomial.
                                         &
                                    %\vspace*{0.1cm}
                                If $\textrm{Vol}_{2k}(X\cap B(R,0))\leq CR^{2k}$ for all $R\in\re^{+}$ with $C$ a positive constant
                                and $X$ an analytic subset, then $X$ is algebraic.\\
                            \midrule
                                    %\vspace*{0.1cm}
                            \begin{center}
                                \textit{Riemann's extension theorem.}
                            \end{center}
                                &
                            \begin{center}
                                \textit{Bishop's generalization of Remmert-Stein's theorem.}
                            \end{center} \\
                            \midrule
                                If $f:(\Omega\setminus E)\subset\co\rightarrow\co$ is a holomorphic function and
                                $E$ is a compact subset of capacity $0$, then $f$ is extendible to a holomorphic
                                function on the whole region $\Omega$.
                                         &
                                %\vspace*{0.1cm}
                                Let $U\subset\con$ be a bounded open
                                subset and let $B\subset U$ be closed with $X\subset U\setminus B$ a purely k
                                dimensional analytic subset such that $B\subset\overline{X}$. If $B$ has
                                capacity $0$ relative to the algebra of analytic functions on $X$ that are
                                continuously extendible to $\overline{X}$, and if it exists
                                $f:U\rightarrow\co^k$ proper on $B$ with $f(B)$ not an open connected subset of
                                $\co^k$, then $\overline{X}\cap U$ is an analytic subset of $U$
                                (see \cite{Bishop}[Theorem 4]).\\
                            \midrule
                            \begin{center}
                                    \textit{Montel's compactness theorem.}
                            \end{center}
                                         &
                            \begin{center}
                            \textit{Lelong's normal currents theorem (consequence of Bishop's sequence theorem) see \cite{Lelong}}
                            \end{center}\\
                            \midrule
                            %\vspace*{0.1cm}
                                    Let $F$ be a family of locally bounded
                                    holomorphic functions $f_i:\Delta\rightarrow\co$.
                                    Then $F$ is a normal family if and only if $F$ is locally uniformly bounded.
                                         &
                                    Let $F$ be a family of analytic subsets and let $t(F)$
                                    be the set of integration currents defined by $F$, suppose $t(F)$
                                    is locally bounded (in terms of the operator norm).
                                    Then $t(F)$ is a normal family if and only if its locally uniformly bounded.\\
                            \bottomrule
            \end{tabular}
\end{table}
\clearpage

\section{Foliations on Kähler manifolds with all leaves compact}
\noindent In this final section, we apply Bishop's sequence theorem (theorem \ref{bishop sequence})
to prove a version of \citep{EMS}[Theorem 1], for the particular case of compact complex foliations on
Kähler manifolds with all leaves compact, also we prove that a holomorphic foliation in a compact Kähler manifold with at least \textbf{one} compact leaf with finite holonomy, then all of its leaves are compact.

\begin{theorem}\label{kahlerEMS}
Let $M$ be a compact connected Kähler manifold
of complex dimension $n$ and $\mathfrak{F}$ a holomorphic foliation
with leaves of complex dimension $d<n$ and with all leaves compact,
then:
\begin{enumerate}
        \item[1] The $2d$-dimensional volume (with
                respect to the Kähler metric) of the leaves is uniformly
                bounded.
        \item[2] The quotient space $M/\mathfrak{F}$ is a complex orbifold and
        the quotient map $\pi:M\to{M/\mathfrak{F}}$ is an analytic map.
         
        \item[3] The singular set $\mathcal S$ of $M/\mathfrak{F}$ is the saturated set that corresponds to the leaves with nontrivial holonomy group (which, by the first proposition, is a finite group). Thus, the singular set $S$ is an analytic subset of $M/\mathfrak{F}$. In particular, since $M$ is connected ($\mathcal S$ is of real codimension greater than one), it follows that the complement ${\mathcal R}:=M\setminus{\mathcal S}$ is connected and the  $2d$-dimensional volume (with
respect to the Kähler metric) of leaves in $\mathcal R$ is constant.
\end{enumerate}
\end{theorem}

First, throughout this theorem and the following results, we are going to make use of the following fundamental theorem, due to W. Thurston (\cite{Thurston}) and called the {\it generalized Reeb stability theorem}:
\begin{theorem}[W. Thurston \cite{Thurston}]\label{Reeb}
Let $\F$ be a smooth foliation of dimension $m$ on the smooth, compact $n$-manifold $M$
($0\leq{m}\leq{n}$) and $\mathcal{L}$ be a compact leaf without holonomy. Then, there exists a \emph{saturated} neighborhood, say $\Omega$, of $\mathcal{L}$ such that $\Omega$ is diffeomorphic to a product $\mathcal{L}\times D^{n-m}$ were 
$D\subset{\mathbb R}^{n-m}$ is an open disk, such that every leaf in $\Omega$ is of the form $\mathcal{L}\times{\left\{p\right\}}$ with $p\in{D}$. Moreover, if $\mathcal{L}$ has \textbf{\emph{finite holonomy}}, there exists a \emph{saturated} tubular neighborhood of 
$\mathcal{L}$, $\Omega$ in which the foliation is a flat disk bundle over $\mathcal{L}$.
In particular, if $M$ is a compact complex manifold (not necessarily Kähler) and $\F$ is a holomorphic foliation, there exists a saturated neighborhood, $\Omega$ of $\mathcal{L}$ such that $\Omega$ is diffeomorphic (although not necessarily biholomorphic) to a product
  $\mathcal{L}\times D$ with $D\subset\co^{q}$ an open polydisk, such that every leaf in $\Omega$ is of the form $\mathcal{L}\times{\left\{p\right\}}$ with $p\in{D}$
\end{theorem}
The proof of the second proposition of the theorem uses the following
{\it Quotient Theorem} of H. Cartan:

\begin{theorem}[H. Cartan \cite{Cartan}]\label{Cartan} Suppose that $X$ is a complex manifold, $G$ is a group of holomorphic automorphisms of $X$ acting properly discontinuously on $X$ (thus $G$ must be discrete), then the quotient $X/G$ has the structure of an analytic space such that the quotient map $X\to{X/G}$ is analytic map. The singular set of $X/G$ corresponds to the set $S$ of points with a nontrivial isotropy subgroup {\it i.e.,} the singular set is
$\pi(S)\subset {X/G}$. In particular, if $X$ is compact $\pi$ is a proper analytic map and
 $\pi(S)$ is an analytic subset of $X/G$, by Grauert's Proper Mapping Theorem \cite{Grauert}
 \end{theorem}
 
 \begin{remark} \label{orbifold} In Theorem \ref{Cartan} we can say more. Since the action is
 proper and discontinuous, the isotropy subgroups of points are finite groups, and
 therefore, $X/G$ is an orbifold, endowed with natural orbifold charts of the form
 $\pi\circ\varphi:U\to{X/G}$, where $\varphi:U\to{M}$ is a local holomorphic chart of $M$.
 \end{remark}

\begin{proof}[Proof of Theorem \ref{kahlerEMS}]
The product structure in Theorem \ref{Reeb} immediately proves the continuity of the volume function (with respect to any Riemannian metric) of the leaves contained in $\Omega$: if $\mathcal{L}_z$ is the leaf through $z\in\Omega$ and $\omega^{2d}(z)$ is the $2d$-dimensional volume form, induced by $g$, on the leaf $\mathcal{L}_z$, then the fuction:

 \begin{equation}\label{oneline}
 p\mapsto  \int_{\mathcal{L}\times{\left\{p\right\}}}\omega^{2d}_{p},
\end{equation}
depends differentiably on $p$ since both the volume form $\omega_p$ and the integrating manifold ${\mathcal{L}\times{\left\{p\right\}}}$ depend differentiably on $p$ (see \cite{Simon}).
   
Now lets define the \emph{leaf volume function} $\nu:M\rightarrow\re^{+}$ as $z\mapsto \textrm{Vol}_{2d}(\mathcal{L}_z)$, where as before, $\mathcal{L}_z$ is the leaf through $z$. Now, Epstein, Millet, and Tischler proved in \cite{EMT} that the set of leaves without holonomy (known as generic leaves)

\begin{equation}\label{generic-leaves}
  H_0=\lbrace x\in M\,|\,\mathcal{L}_x\text{ has trivial holonomy}\rbrace,
\end{equation}
\noindent is a dense $G_\delta$ set. Moreover, the stability theorem implies that the set of leaves with trivial holonomy is also open by the previously discussed product structure around generic leaves.

We show that $\nu$ is locally constant in the set of generic leaves, let $z\in H_0$ and let $\mathcal{L}_z$ be the leaf containing $z$, since $\mathcal{L}_z$ has zero holonomy, by the generalized Reeb stability theorem, there exists a tubular neighborhood of $\mathcal{L}_z$, say $\Omega$, which is diffeomorphic to a disk bundle over $\mathcal{L}$, with fiber $D$, the open unit disk in $\co^{n-d}$, and such that $\Omega$ is a saturated open subset of $M$. Since every leaf in $\Omega$ is homologous to $\mathcal{L}_z$, by Wirtinger's inequality (Collorary \ref{wirtinger}), all leaves in $\Omega$ have the same volume; therefore $\nu$ is locally constant in $H_0$.

%%let $\lbrace z_n\rbrace\subset H_0$ be a sequence, without loss of generality, suppose that each $z_i$ is on a different leaf $\mathcal{L}_{z_i}$ of $\mathfrak{F}$ and suppose that $z_n\rightarrow z\in H_0$.
%%Since all the leaves are compact it is clear that $\mathcal{L}_{z_{i}}\rightarrow\mathcal{L}$ for the Hausdorff metric, where $\mathcal{L}\subset M$ is a non-empty closed set.
%Now,  since $z\in\Omega$ and $z_n\rightarrow z$ there is a large enough positive integer $N$ such that all leaves $\mathcal{L}_{z_k}$ have the same volume for $k>N$. Therefore, by theorem \ref{bishop sequence}, $\mathcal{L}$ is an analytic subset of $\Omega$ of complex dimension $d$ with its volume equal to $\lim_{n\rightarrow\infty} \textrm{Vol}_{2d}(\mathcal{L}_{z_n})$. Now, tangency to $\mathfrak{F}$ is defined locally by the null space of $d$ holomorphic 1-forms, by Hausdorff convergence, this tangency is preserved on the limit, so $\mathcal{L}$ is tangent to $\mathfrak{F}$ and therefore $\mathcal{L}=\mathcal{L}_z$, therefore $\nu$ is constant on $\Omega$.

Now, $\nu$ is not continuous in general but rather lower semicontinuous, but we will prove more than that;
the volume function $\nu$ is, in fact, \textbf{\emph{discretely}} lower-semicontinuous, meaning that for any
$n\in\zah^+$, $z\in M$ and $\epsilon\in\re^+$, there is a small enough
neighborhood of $z$ such that either
\[
  \nu(y)>n\,\nu(z)\hspace{0.2cm}\text{or}\hspace{0.2cm}|\nu(y)-k\,\nu(z)|<\epsilon \hspace{0.2cm}\text{for some }\,k\in\lbrace 1,\ldots,n\rbrace.
\]
We note here that semicontinuity can be proved by showing that the leaf space $M/\mathfrak{F}$ is Hausdorff (see \citep{EMS}[p. 20]).
Now to prove discrete lower-semicontinuity, we work locally; an easy consequence of the generalized Reeb Stability Theorem \ref{Reeb} is that for leaves with finite holonomy, there exists a tubular neighborhood of a leaf $\mathcal{L}_z$, say $W$ and a bundle retraction $\rho: W\rightarrow\mathcal{L}_z$ with $\rho^{-1}(x)$ homeomorphic to a disk.
For every leaf $\mathcal{L}_y$, the restriction $\rho|_{W\cap\mathcal{L}_y}:(W\cap\mathcal{L}_y)\rightarrow\mathcal{L}_z$ is a codimension zero submersion, if $y$ is sufficiently close to $z$, then $\mathcal{L}_y\subset W$ and also the image under $\rho$ of the leaf $\mathcal{L}_y$ covers all of $\mathcal{L}_z$.
Therefore, by compactness and analyticity of the leaves, $\mathcal{L}_y$ is a finitely sheeted covering space of $\mathcal{L}_z$ with covering transformation $\rho|_{\mathcal{L}_y}$ which proves the discrete lower-semicontinuity of $\nu$ for the set of leaves with finite \textbf{bounded} holonomy.
With this, the only thing left in order to show semicontinuity is to show that the set where $\nu$ is not locally bounded is empty, proving this is equivalent to showing that all leaves have finite bounded holonomy. The set where $\nu$ is not locally bounded is also known as the ``bad set'':
\[
    B:=\lbrace x\in M\,|\, \nu\,\text{is not bounded in a neighborhood of }x\rbrace,
\]
and is a saturated compact set of codimension greater than or equal to 2 (see \cite{Epstein}).
By definition $B$ is the union of leaves for which $\nu$ is not bounded around them, therefore its a saturated set, also, $B$ is a closed set, since for any convergent sequence of points in $B$, \(x_n\rightarrow x\) there is another sequence of points $\{x_{n_k}\}$ arbitrarily close to $x_n$ such that $\{\nu(x_{n_k})\}$ is unbounded.
Therefore since $\{x_n\}$ is arbitrarily close to $x$, by a diagonal argument there is a sequence $\{x_{n_l}\}$ arbitrarily close to $x$ such that $\nu(x_{n_l})$ is unbounded, therefore $B$ is a saturated closed (compact) subset.
Now let $\{x_n\}$ be a sequence in $B$, suppose without loss of generality that every \(x_i\) is on a different leaf, since $B$ is a closed compact saturated set and all leaves are compact, there is a finite covering of $B$ by open saturated sets $U_j$ such that for any convergent subsequence $\{x_{n_k}\}$, their corresponding leaves $\{\mathcal{L}_{n_k}\}$ converge $\mathcal{L}_{n_{k}}\to\mathcal{L}$ with respect to the Hausdorff metric in $U_j\cap B$ for some $j$, where $\mathcal{L}\subset M$ is a non-empty closed set. Note that the sub sequence of points must converge also \(x_{n_k}\to{x}\in U_j\subset B\).
%%%%%%%%%%%%%%%%%%%%%%%%%%%%%%
%% This argument needs work:%%
%%%%%%%%%%%%%%%%%%%%%%%%%%%%%%
Now, $U_j\cap B$ is a saturated open set such that eventually every leaf $\mathcal{L}_{n_k}$ stays in $U_j$, since $x\in U_j$ and $x_{n_k}\to{x}$, therefore, we may assume (maybe by shrinking $U_j$ a little bit) that eventually every leaf will be homologous to each other on $U_j$, and by corollary \ref{wirtinger}, there exist a large enough positive integer $N$ such that all leaves $\mathcal{L}_{n_k}$ have the same volume for $k>N$. Therefore, by theorem \ref{bishop sequence}, $\mathcal{L}$ is an analytic subset of $U_j$ of complex dimension $d$.
Now, tangency to $\mathfrak{F}$ is defined locally by the null space of $d$ holomorphic 1-forms, by Hausdorff convergence, this tangency is preserved on the limit (being a closed condition), so $\mathcal{L}$ is tangent to $\mathfrak{F}$ and therefore $\mathcal{L}=\mathcal{L}_x$ the leaf through $x$.
Since $\{x_{n}\}$ was arbitrary and we have showed that every leaf on $B$ has a convergent sequence of leaves near them with bounded volumes in a neighborhood of $B$, therefore $B=\emptyset$.

%%By the local boundeness of $\nu$ in $M\setminus B$ and Ehresmann's structure theorem, the set $(M\setminus B)\setminus H_0$ is a countable union of smooth manifolds of codimension $\geq 2$. Let $\lbrace \Sigma_n\rbrace_{n\in\nat}$ be the components of $M\setminus B$, then for each $\sigma_k$ we can find a sequence $\rbrace\mathcal{L}_m\rbrace$ of generic leaves such that $\mathcal{L}_m\rightarrow\Sigma_k$

The second assertion follows easily from Theorem \ref{Reeb}, since every leaf has finite and bounded holonomy, for ever leaf $\mathcal{L}$ we have an open saturated set $U$ biholomorphic to $\mathcal{L}_z\times D$ so $M/\mathfrak{F}$ is locally homeomorphic to $D$.
Furthermore, $M/\mathfrak{F}$ is Hausdorff since every leaf is compact, so if $\mathcal{L}_1$ and $\mathcal{L}_2$ are two distinct leaves, then, there are $\lbrace\epsilon_1,\epsilon_2\rbrace\subset\re^+$ such that the sets
\[
  D_i:=\lbrace z\in M\,|\, d_H(\mathcal{L}_i,z)<\epsilon_i\rbrace,\hspace{0.2cm} i\in\lbrace1,2\rbrace,
\]
\noindent are disjoint, so intersecting with a saturated tubular neighborhood
of $\mathcal{L}_{i}$, we have that $M/\mathfrak{F}$ is Hausdorff. Finally, if
$\mathcal{L}$ has non-trivial holonomy, then by the boundedness of the volume function, the
holonomy group $H(\mathcal{L})$ is finite (see \citep{EMS} [p. 20]) and
$M/\mathfrak{F}$ is locally homeomorphic to $D/H(\mathcal{L})$, where
$\mathcal{L}$ has a tubular neighborhood  homeomorphic to $\mathcal{L}\times D$.
On the other hand, H. Cartan's Theorem \ref{Cartan} implies that $D/H(\mathcal{L})$ is an analytic set and the quotient map $D\to{D/H(\mathcal{L})}$ is analytic. 

The arguments above prove
that the quotient $M/\mathfrak{F}$ is an analytic set, and the quotient map
$\pi:M\to{M/\mathfrak{F}}$ is analytic. It also follows from Theorem \ref{Cartan} that the union of the set of leaves with nontrivial holonomy is an analytic subset of $M$, and, as noted in Remark \ref{orbifold}, $M/\mathfrak{F}$ is a complex orbifold.
This proves items (2) and (3) of Theorem \ref{kahlerEMS}
\end{proof}
We finish this paper with a proof of a result of Jorge Vitorio Pereira, which guarantees the compactness of all leaves in a compact Kähler manifold provided that it is known that at least \textbf{one} of the leaves is compact and has finite holonomy (see \cite{Pereira}[Theorem 1]).
\begin{theorem}\label{compact-leaf}
Let $M$ be a compact connected Kähler manifold of dimension $n$ and $\mathfrak{F}$ be a holomorphic foliation of codimension $q<n$ for which there is at least one compact leaf with finite holonomy, then all leaves are compact with finite holonomy or equivalently their volumes are uniformly bounded and the leaf space is Hausdorff.
\end{theorem}
\begin{proof}
  Let us denote by $\mathcal L_{0}$ to the compact leaf with finite holonomy of $\mathfrak F$ and define the following set
\[
    \mathcal{U}:=\{x\in M\,|\, \text{ the leaf through $x$ is compact with finite holonomy}\}.
\]
Clearly $\mathcal{L}\subset\Omega\subset\mathcal{U}$ and therefore $\mathcal{U}$ is a non-empty saturated set.
Also by the same argument as before every leaf in $\mathcal{U}$ has a tubular saturated neighborhood comprised of
compact leaves with finite holonomy, therefore for every point $z\in\mathcal{U}$ there is an open saturated
neighborhood $W$ of $z$ such that $z\in W\subset\mathcal{U}$ i.e. \(\mathcal{U}\) is an \emph{open} saturated set.
Note that $\mathcal{U}$ is the union of all saturated open sets with all its leaves compact and with finite holonomy;
another way to describe it is as the maximal open set of saturated sets with all its leaves compact and with finite holonomy.
We show that this set is also closed to finish the proof. Let $\{z_{n}\}\subset\mathcal{U}$ be a convergent sequence
with limit $z_{n}\rightarrow z_{0}$, since $\{z_{n}\}$ is convergent, in particular is a Cauchy sequence.
Moreover, let $\mathcal{L}_{n}$ be the leaf through $z_{n}$ and $\mathcal{L}_{0}$
the leaf through $z_{0}$, since $z_k\in\mathcal{U}$, every $\mathcal{L}_k$ is compact with finite holonomy, to show
$z_{0}\in\mathcal{U}$ we need to show $\mathcal{L}_{0}$ is compact and has finite holonomy, first since $\{z_{k}\}$ is a Cauchy
sequence, the leaves $\{\mathcal{L}_{k}\}$ are eventually arbitrarily close between them. Since all of them are compact
submanifolds of a Kähler manifold, they minimize the volume of their homology classes respectively, therefore, there is a
$N\in\nat$ such that for all \(\{\mathcal{L}_{k}\,|\,k\geq N\}\) are in the same homological class, therefore eventually the
volume of $\{\mathcal{L}_k\}$ is constant.
Now for any foliated chart $V$ around $z_{0}$, let $\mathcal{L}_{k}\cap V$ be a sequence in the intersection
$V\cap\mathcal{U}$, since $z_{0}$ is a limit point of $\mathcal{U}$ and $\mathcal{U}\cap V\neq\emptyset$, by convergence
there is a $N\in\nat$ such that for $k\geq N$, $z_k\in V$. Moreover, since $V$ is a foliated chart,
$V$ is biholomorphic to a product $W_q\times W_{n-q}$ and every leaf $\{\mathcal{L}_k\cap V\}_{k\geq N}$ is the set
$\varphi^{-1}(W_q\times\{z_k\})$, here we abused the notation a little bit by identifying the points of the
convergent sequence with their corresponding transversal coordinates.
Its clear that \hbox{$V\cap\mathcal{L}_{k}\rightarrow V\cap\mathcal{L}_{0}$} as closed subsets of $\mathcal{U}\cap V$,
since $z_k\rightarrow z_0$ and the plaques $W_q\times\{z_k\}$ clearly converge to $W_q\times\{0\}$ by continuity of $\varphi$.
Therefore, since \(\mathcal{L}_k\) are compact connected, and with finite (linear) holonomy and bounded volume,
by Bishop's sequence theorem (theorem \ref{bishop sequence}), \(\{\mathcal{L}_k\}\) converges to a closed (compact) pure-dimensional analytic subset of $\mathcal{L}_{0}$ that is tangent to $\mathfrak{F}$ at $z_0$ therefore
\hbox{$\mathcal{L}_{k}\rightarrow\mathcal{L}_{0}$} since leaves are connected by definition.
This means that $\mathcal{L}_{0}$ is compact analytic set with finite holonomy since the volume of the leaves is bounded and nearby leaves cover $\mathcal{L}_{0}$,
therefore $\mathcal{U}=M$ since its both open and closed and $M$ is connected.
\end{proof}

\subsection*{\bf Compliance with ethical standards}
 {\bf Conflict of interest.} The authors declare that they do not have conflict of interests.

\end{document}